\documentclass[12pt]{amsart}

\usepackage{amscd,mathrsfs,epsfig,amsthm,amssymb,amsmath}
\usepackage[all]{xy} 
\usepackage[T1]{CJKutf8}
\usepackage[sans]{dsfont}  
\usepackage{hyperref}
\usepackage{tabularx}
\usepackage{color}

\usepackage{graphicx} 
\usepackage[maxnames=4]{biblatex}
\newtheorem{theorem}{Theorem}[subsection]

\newtheorem{definition}[theorem]{Definition}
\newtheorem{proposition}[theorem]{Proposition}
\newtheorem{remark}[theorem]{Remark}

\newtheorem*{thma}{Theorem A}
\newtheorem*{thmb}{Theorem B}

\allowdisplaybreaks[2]

\def\calA{{\mathcal A}}

\def\calC{{\mathcal C}}
\def\calD{{\mathcal D}}

\def\calJ{{\mathcal J}}

\def\End{\mathop{\rm End}\nolimits} 
\def\Ext{\mathop{\rm Ext}\nolimits}

\def\Hom{\mathop{\rm Hom}\nolimits} 
\def\Nat{\mathop{\rm Nat}\nolimits}

\def\lim{\mathop{\varinjlim}\nolimits}
 
\def\Ob{\mathop{\rm Ob}\nolimits} 
\def\Mor{\mathop{\rm Mor}\nolimits}

\def\PSh{\mathop{\rm PSh}\nolimits}
\def\Sh{\mathop{\rm Sh}\nolimits}

\def\cod{\mathop{\rm cod}}
\def\dom{\mathop{\rm dom}}

\DeclareMathOperator{\C}{\mathbf{C}}

\DeclareMathOperator{\Mod}{\mathfrak{M}od-}
\DeclareMathOperator{\PMod}{\mathfrak{M}od^{\it pr}-}
\DeclareMathOperator{\Ab}{\mathbf{Ab}}
\DeclareMathOperator{\rMod}{Mod-}

\begin{document}

\title{Torsion pairs in categories of modules on ringed finite sites}

\author{Mawei Wu}
\address{School of Mathematics and Statistics, Lingnan Normal University, Zhanjiang, Guangdong 524048, China}
\email{wumawei@lingnan.edu.cn}






\begin{abstract}
Let $\calC$ be a small category. In this paper, we mainly study the category of modules $\Mod\mathfrak{R}$ on ringed sites $(\C,\mathfrak{R})$. We firstly reprove \cite[Theorem A]{WX23}, then we characterize $\Mod\mathfrak{R}$ in terms of the torsion modules on $Gr(\mathfrak{R})$, where $Gr(\mathfrak{R})$ is the linear Grothendieck construction of $\mathfrak{R}$. Finally, we investigate the hereditary torsion pairs, TTF triples and Abelian recollements of $\Mod\mathfrak{R}$. When $\calC$ is finite, the complete classifications of all these are given respectively.    
\end{abstract}

\maketitle

\tableofcontents

\section{Introduction}

Let $\calC$ be a small category. The category $\calC$ equipped with a Grothendieck topology $\calJ$ is called a \emph{site}, written as $\C=(\calC, \calJ)$. If $\mathfrak{R}$ is a sheaf of rings on $\C$, we will denote by $\Mod \mathfrak{R}$ the category of sheaves of modules on a ringed site $(\C, \mathfrak{R})$, which is the main object that we will study in this paper.

Let $\mathcal{A}$ be a small preadditive category. In 2021, Parra, Saor{\'\i}n and Virili \cite{PSV21} had investigated the category of \emph{right} modules $\rMod \calA$ over $\mathcal{A}$. Using the \emph{linear Grothendick topologies} (see Definition \ref{lintop}), they extended Gabriel's classical bijection (there is a one-to-one correspondence between Gabriel topologies in a unitary and associative ring $k$ and hereditary torsion classes in $\rMod k$ \cite{Gab62}). More explicitly, they obtained a one-to-one correspondence between linear Grothendieck topologies on the preadditive category $\mathcal{A}$ and hereditary torsion pairs in $\rMod \calA$. Besides, they also got some characterizations of the TTF triples and Abelian recollements of $\rMod \calA$. Our main motivation is trying to get some analogous results in the category of sheaves of modules $\Mod \mathfrak{R}$ on ringed sites.    

Given a presheaf of unital $k$-algebras $\mathfrak{R}:\calC^{\rm op} \to k\mbox{\rm -Alg}$, one can define the \emph{linear Grothendick construction} $Gr(\mathfrak{R})$ of it (see Definition \ref{lingrocon}). The linear Grothendick construction $Gr(\mathfrak{R})$ is a preadditive category, so one can put some linear Grothendieck topologies on it. A result  of Howe (see Theorem \ref{How81}) says that the sheaves of $\mathfrak{R}$-modules can be viewed as the Abelian group-valued sheaves on $Gr(\mathfrak{R})$. Thanks to this observation, we give another proof of the \cite[Theorem A]{WX23} (see the proof of the Theorem \ref{reproof}). Moreover, together with a Proposition (see \cite[Proposition 2.7]{Low16}) of Lowen, we obtain the following result, which says that the sheaves of $\mathfrak{R}$-modules can be characterized in terms of the torsion modules on $Gr(\mathfrak{R})$. Here, we denote by $\mathcal{B}^{\perp}$ the \emph{right perpendicular category} of $\mathcal{B}$, see the paragraph after Definition \ref{serre} for its precise definition, and by ${\rm Tors}(\mathcal{A}, \calJ) \subseteq \rMod \calA$ the \emph{full subcategory of torsion modules} (see Definition \ref{tormod}). 

\begin{thma} (Theorem \ref{tor})
Let $\calC$ be a small category and let $\mathfrak{R}:\calC^{\rm op} \to k\mbox{\rm -Alg}$ be a sheaf of unital $k$-algebras on a site $\C=(\calC, \calJ)$. Then we have the following equivalence  
    $$
  \Mod \mathfrak{R} \simeq {\rm Tors}(Gr(\mathfrak{R}), \calJ')^\perp,
    $$
for some linear Grothendieck topology $\calJ'$ on $Gr(\mathfrak{R})$.    
\end{thma}

Let $\mathfrak{R}:\calC^{\rm op} \to k\mbox{\rm -Alg}$ be a sheaf of unital $k$-algebras, in the category of sheaves of modules $\Mod \mathfrak{R}$ on \emph{ringed finite sites} (see Defintion \ref{finsit} (2)), all the sheaves of $\mathfrak{R}$-modules can be viewed as the presheaves of $\mathfrak{R}|_{\calD}$-modules for some strictly full subcategory $\calD$ of $\calC$ (see Proposition \ref{shispre}), where $\mathfrak{R}|_{\calD}$ is the restriction of $\mathfrak{R}$ to $\calD$. Combining this with Howe's result (see Theorem \ref{How81}) and the work of Parra, Saor{\'\i}n and Virili \cite{PSV21}, we can classify the \emph{hereditary torsion pairs}, \emph{(split) TTF triples} and \emph{Abelian recollements} in the category of sheaves of modules on ringed finite sites.   

\begin{thmb} (Theorem \ref{htp}, \ref{ttf}, \ref{sttf} and \ref{ar})
The complete classifications of the hereditary torsion pairs, (split) TTF triples and Abelian recollements in the category of modules $\Mod\mathfrak{R}$ on ringed finite sites are given, respectively.  	
\end{thmb}

This paper is organized as follows. In Section \ref{prelim} we recall some basics about category theory and sheaf theory, as well as the definitions of the linear Grothendieck constructions and the skew category algebras. In Section \ref{torperp}, we firstly reprove the \cite[Theorem A]{WX23}, and then we give a characterization of the sheaves of $\mathfrak{R}$-modules in terms of the torsion modules on $Gr(\mathfrak{R})$. Finally, the torsion pairs, TTF triples and Abelian recollements of the category of modules on ringed finite sites are studied in Section \ref{classification}.

\section{Preliminaries} \label{prelim}

In this section we will recall some basic concepts and results on category theory and sheaf theory. There are two kinds of sheaf theoretic languages, namely, in the \emph{non-additive} context and \emph{additive} context. In this paper, we will encounter both of them. For the non-additive context sheaf theory, our main references are \cite{AGV72,StaPro,KS06,MM92}. For the additive context sheaf theory, one can see \cite{Low04, Low16}. Here we only recall some necessary ingredients.

\noindent\textbf{Convention}: We shall use curly upper-case letters $\mathcal{A}, \mathcal{B}, \mathcal{X}, \mathcal{Y}$ etc for categories (underlying sites); lower-case letters $a, b, x, y$ etc for the objects of categories; lower-case letters $f, g, h$ etc for the morphisms of categories. Functors among these categories will be denoted by Greek letters $\Phi, \Psi, \alpha, \beta$ etc. Presheaves and sheaves are written as Fraktur letters $\mathfrak{F}, \mathfrak{G},\mathfrak{M}, \mathfrak{R}$ etc. If $f$ is a morphism in $\calC$, we denote its domain and codomain by ${\rm dom}(f)$ and ${\rm cod}(f)$ respectively. We will denote by $\Hom_{\mathcal{A}}(x, y)$ or $\mathcal{A}(x, y)$ the morphism set between objects $x$ and $y$. 

In this paper, we shall denote by $k$ a commutative ring with identity, and by $k$-Alg the category of unital associative algebras and unital $k$-algebra homomorphisms. All algebras and algebra morphisms will be unital. All modules, either over an algebra or a ringed site, are \emph{right} modules. We shall denote by $\Ab$ the category of Abelian groups, and by $\Mod \mathfrak{R}$ the category of sheaves of modules on a ringed site $(\calC, \mathfrak{R})$. In order to specify the trivial Grothendieck topology on it, we will denote the category of presheaves of modules on a small category $\calC$ by $\PMod \mathfrak{R}$. The category of presheaves of Abelian groups over a preadditive category $\mathcal{A}$, which is called the category of right modules over $\mathcal{A}$ in this paper, will be written as $\rMod \calA$.

\subsection{Categories and sheaves}

\subsubsection{Categories}

In this subsection, we will recall some basic definitions about category theory, and mostly about the preadditive categories.

\begin{definition}
Let $\calC$ be a small category. It is said to be \emph{finite} if $\Mor \calC$ is finite, and it is \emph{object-finite} if $\Ob \calC$ is finite.    
\end{definition}

\begin{definition}
An endomorphism $f \in \Hom_{\calC}(x, x)$ is called an \emph{idempotent endomorphsim} if it is an idempotent (i.e. $f^2=f$).   
\end{definition}

\begin{definition}
A category $\mathcal{A}$ is \emph{preadditive} if each morphism set $\Hom_{\mathcal{A}}(x, y)$ is endowed with the structure of an Abelian group.   \end{definition}

Given a small preadditive category $\mathcal{A}$, one can associate it with the following two categories: the \emph{additive closure} $\widehat{\mathcal{A}}$ and the \emph{Cauchy completion} $\widehat{\mathcal{A}}_{\oplus}$ of it (namely, the idempotent completion  of the additive closure of $\mathcal{A}$, see \cite[Section 1.1]{PSV21} for their precise definitions).

\begin{definition} (\cite[pages: 1148, 1150]{PSV21})
\begin{enumerate}    
\item  Let $\mathcal{A}$ be a small preadditive category, a \emph{(two-sided) ideal} of $\mathcal{A}$ is a subfunctor 
$$
\mathcal{I}(-, -) \leq \mathcal{A}(-, -): \mathcal{A}^{\rm op} \times \mathcal{A} \to \Ab. 
$$
That is, for any $f \in \Hom_{\mathcal{I}}(x, y)$ and $l \in \Hom_{\mathcal{A}}(y, y')$, $r \in \Hom_{\mathcal{A}}(x', x)$, $l \circ f \circ r \in \Hom_{\mathcal{I}}(x',y')$, which is a subgroup of $\Hom_{\mathcal{A}}(x',y')$. 

\item Let $\mathcal{I}, \mathcal{K}$ be two ideals of $\mathcal{A}$, the \emph{product ideal} $\mathcal{I} \cdot \mathcal{K}$ is defined as follows:
$$
(\mathcal{I} \cdot \mathcal{K})(x, y):=\left\{ \sum_{i=1}^{n} g_i \circ f_i \ |\ g_i \in \Hom_{\mathcal{I}}(z_i, y), f_i \in \Hom_{\mathcal{K}}(x, z_i)   \right\}.
$$

\item An ideal $\mathcal{I}$ is said to be \emph{idempotent} if $\mathcal{I} \cdot \mathcal{I}=\mathcal{I}$.
\end{enumerate}
\end{definition} 

Given a unitary ring, one can always construct a two-sided ideal generated by a given family of elements. The analogous construction in the preadditive categories is given as follows.

\begin{definition} (\cite[Definition 2.2]{PSV21})   
Let $f: x \to y$ be a morphism in $\mathcal{A}$ and let $\mathcal{M}$ be a set of morphisms of $\mathcal{A}$. Then
\begin{enumerate}
    \item \emph{the (two-sided) ideal of $\mathcal{A}$ generated by $f$}, denoted by $\mathcal{A}f\mathcal{A}: \mathcal{A}^{\rm op} \times \mathcal{A} \to \Ab$, is the subfunctor of $\mathcal{A}(-, -)$ such that $\mathcal{A}f\mathcal{A}(a, b)$ is the subgroup of $\mathcal{A}(a,b)$ generated by compositions $g \circ f \circ h$, where $h \in \mathcal{A}(a, x)$ and $g \in \mathcal{A}(y, b)$;
    \item \emph{the (two-sided) ideal of $\mathcal{A}$ generated by $\mathcal{M}$}, denoted by $\mathcal{A}\mathcal{M}\mathcal{A}: \mathcal{A}^{\rm op} \times \mathcal{A} \to \Ab$, is the sum
    $$
    \mathcal{A}\mathcal{M}\mathcal{A}:=\sum_{f \in \mathcal{M}} \mathcal{A}f\mathcal{A}.
    $$
That is, $\mathcal{A}\mathcal{M}\mathcal{A}(a, b)=\sum_{f \in \mathcal{M}} \mathcal{A}f\mathcal{A}(a, b)$, where the sum is the sum of subgroups of the Abelian group $\mathcal{A}(a, b)$. 
\end{enumerate}
\end{definition}

Now, let us recall the definition of the center of a preadditive category. 

\begin{definition} (\cite[page: 1147]{PSV21})   
Let $\mathcal{A}$ be a small preadditive category, the \emph{center} $Z(\mathcal{A})$ of $\mathcal{A}$ is the ring of self-natural transformations of the identity functor ${\rm Id}_{\mathcal{A}}$, that is,
$$
Z(\mathcal{A}):={\mathcal{F}un}(\mathcal{A}, \mathcal{A})({\rm Id}_{\mathcal{A}}, {\rm Id}_{\mathcal{A}}).
$$
where ${\mathcal{F}un}(\mathcal{A}, \mathcal{A})$ is the $\Ab$-enriched functor category.
\end{definition}

Let $\mathcal{A}$ be a small preadditive category. A \emph{right module $\mathfrak{M}$ over $\mathcal{A}$} is an (always additive) functor $\mathfrak{M}:\mathcal{A}^{\rm op} \to \Ab$. We denote by $\rMod \calA$ \emph{the category of right $\mathcal{A}$-modules} (also called the category of presheaves of Abelian groups on a preadditive category $\mathcal{A}$). Given a class $\mathcal{S}$ of $\mathcal{A}$-modules and an $\mathcal{A}$-module $\mathfrak{M}$, one can construct a submodule ${\rm tr}_{\mathcal{S}}(\mathfrak{M})$ of $\mathfrak{M}$ such that any map $\mathfrak{S} \to \mathfrak{M}$, with $\mathfrak{S} \in \mathcal{S}$, factors through the inclusion ${\rm tr}_{\mathcal{S}}(\mathfrak{M}) \to \mathfrak{M}$.
  
\begin{definition} (\cite[Definition 1.8]{PSV21})
Let $\mathcal{S}$ be a class of right $\mathcal{A}$-modules and $\mathfrak{M}$ a right $\mathcal{A}$-module, then the sum of the submodules of $\mathfrak{M}$ of the form ${\rm Im}(f)$, for some morphism $f: \mathfrak{S} \to \mathfrak{M}$ in $\rMod \calA$, with $\mathfrak{S} \in \mathcal{S}$, is called \emph{the trace of $\mathcal{S}$ in $\mathfrak{M}$} and denoted by ${\rm tr}_{\mathcal{S}}(\mathfrak{M})$.
\end{definition}

For a class $\mathcal{S}$ of right $\mathcal{A}$-modules, one can define a (two-sided) ideal ${\rm tr}_{\mathcal{S}}(\mathcal{A})$ of $\mathcal{A}$, which is called the trace of $\mathcal{S}$ in $\mathcal{A}$.

\begin{definition} (\cite[Section 1.3 and Section 2.2]{PSV21})
Let $\mathcal{A}$ be a preadditive category, $\mathcal{A}(-, -): \mathcal{A}^{\rm op} \times \mathcal{A} \to \Ab$ the regular $\mathcal{A}$-bimodule and $\mathcal{S}$ a class of right $\mathcal{A}$-modules. The assignment $x \mapsto {\rm tr}_{\mathcal{S}}(\mathcal{A}(-, x))$ defines a subfunctor of the functor $x \mapsto \mathcal{A}(-, x)$. Viewing this subfunctor as an $\mathcal{A}$-bimodule, then this $\mathcal{A}$-bimodule is called \emph{the trace of $\mathcal{S}$ in $\mathcal{A}$}, and denoted by ${\rm tr}_{\mathcal{S}}(\mathcal{A})$, it is a two-sided ideal of $\mathcal{A}$.
\end{definition}

Now let us recall the definition of the localizing Serre subcategory.

\begin{definition} (\cite[page: 183]{Low16}) \label{serre}
A full subcategory $\mathcal{B}$ of a Grothendieck Abelian category $\mathcal{A}$ is called a \emph{localizing Serre subcategory} if it is closed under coproducts, subquotients and extensions.
\end{definition}

Let $\mathcal{B}$ be a localizing Serre subcategory of a Grothendieck Abelian category $\mathcal{A}$, then the \emph{Serre quotient} $\mathcal{A} / \mathcal{B}$ exists and it is equivalent to the \emph{right perpendicular (or orthogonal) category}
$$
\mathcal{B}^{\perp}:=\{a \in \mathcal{A}\ | \Hom_{\mathcal{A}}(b,a)=0=\Ext^1_{\mathcal{A}}(b, a)\ \mbox{for all}\ b\in \mathcal{B} \}.
$$

\subsubsection{Non-additive sheaf theory}

In this subsection, we will recall some definitions of non-additive sheaf theory. For the references, one can see \cite{AGV72,StaPro,KS06,MM92}.

In order to reformulate the notion of a Grothendieck topology, Giraud first introduced the concept of sieves which is a categorical analogue of a collection of open subsets of a fixed open set in topology. 

\begin{definition}
Let $\calC$ be a small category.  A \emph{sieve} $S$ on $x \in \Ob \calC$ is a subfunctor of the representable functor $\Hom_{\calC}(-,x)$.  
\end{definition}
 
A sieve on $x$ can be identified with a set (still written as $S$) of morphisms with codomain $x$, satisfying the condition that if $g \in S$ and $gf$ exists then $gf \in S$. With the notion of sieves, we can now give the definition of the Grothendieck topology. 

\begin{definition} \label{grotop} A \emph{Grothendick topology} on a category $\calC$ is a function $\mathcal{J}$ which assigns to each object $x$ of $\calC$ a non-empty collection $\mathcal{J}(x)$ of sieves on $\calC$, in such a way that
	\begin{enumerate}
		\item the maximal sieve $\Hom_{\calC}(-,x)$ is in $\mathcal{J}(x)$;
		
		\item (stability axiom) if $S \in \mathcal{J}(x)$, then $f^*(S)=\{g\ |\ fg \in S\} \in \mathcal{J}(y)$ for any arrow $f: y \rightarrow x$;
		
		\item (transitivity axiom) if $S_1 \in \mathcal{J}(x)$, and $S_2$ is any sieve on $x$ such that $f^*(S_2) \in \mathcal{J}(y)$ for all $f: y \rightarrow x$ in $S_1$, then $S_2 \in \mathcal{J}(x)$.
	\end{enumerate}
\end{definition}

Any sieve in $\mathcal{J}(x)$ is called a \emph{covering sieve} on $x$.  A small category $\calC$
equipped with a Grothendieck topology $\mathcal{J}$ is called a \emph{site} $\textbf{C} = (\calC, \mathcal{J})$.

We will denote by $\PSh(\calC)$ \emph{the category of presheaves of sets on} $\calC$, namely, the category of all the contra-variant functors from $\calC$ to the category of sets. Given a Grothendieck topology, one can consider sheaves on a site $\C$. A concise definition of a sheaf is the following. 

\begin{definition} \label{sheaf}
A presheaf $\mathfrak{F} \in \PSh(\calC)$ is a ($\mathcal{J}$-)\emph{sheaf of sets} if, for every
$x \in \Ob \calC$ and every $S \in \calJ(x)$, the inclusion $S \hookrightarrow \Hom_{\calC}(-, x)$ induces an isomorphism
$$
\Nat(\Hom_{\calC}(-, x), \mathfrak{F}) \overset{\cong} \longrightarrow \Nat(S, \mathfrak{F}).
$$
\end{definition}

The category of sheaves of sets on $\mathbf{C} = (\calC, \mathcal{J})$ is a full subcategory of $\PSh(\calC)$, we shall denote it by $\Sh(\mathbf{C})$. 

In this paper, we will consider not only the (pre)sheaves of sets, but also the (pre)sheaves with suitable algebraic structures. Let $(\mathcal{B}, s)$ be a type of \emph{algebraic structure} (see \cite[Section 6.15]{StaPro} for its precise definition) and let $\C$ be a site, we will denote by $\PSh(\calC, \mathcal{B})$ (resp. $\Sh(\C, \mathcal{B})$) the category of presheaves (resp. sheaves) with values in $\mathcal{B}$ on a category $\calC$ (resp. on a site $\C$).

\begin{definition} (\cite[Section 7.44, tag 00YR, page: 338 ($\alpha$)]{StaPro})
A presheaf with values in $\mathcal{B}$ is a \emph{sheaf} if its underlying presheaf of sets is a sheaf.
\end{definition}

If $\mathcal{B}$ is the category of Abelian groups (resp. rings, $k$-modules, $k$-algebras for a fixed unital ring $k$), then we will obtain the definition of \emph{a sheaf of Abelian groups (resp. rings, $k$-modules and $k$-algebras}) on a site $\C$.

Before introducing the main object that we want to study, the following definitions are also needed.

\begin{definition} \label{finsit} Let $\C$ be a site.
\begin{enumerate}
\item A \emph{ringed site} is a pair $(\C, \mathfrak{R})$ where $\C$ is a site and $\mathfrak{R}$ is a sheaf of rings on $\C$. The sheaf $\mathfrak{R}$  is called the \emph{structure sheaf} of the ringed site. 
\item A ringed site $(\C, \mathfrak{R})$ is said to be \emph{finite} if $\calC$ is finite.
\end{enumerate}	
\end{definition}

\begin{definition} (\cite[Section 18.9, tag 03CT]{StaPro})
Let $\calC$ be a category, and let $\mathfrak{R}$ be a presheaf of $k$-algebras on $\calC$. 
\begin{enumerate}
    \item A \emph{presheaf of $\mathfrak{R}$-modules} is given by a presheaf  $\mathfrak{M}$ of $k$-algebras together with a map of presheaves of sets
    $$
    \mathfrak{R} \times \mathfrak{M} \to \mathfrak{M}
    $$
    such that for every object $x$ of $\calC$ the map $\mathfrak{R}(x) \times \mathfrak{M}(x) \to \mathfrak{M}(x)$ defines the structure of an $\mathfrak{R}(x)$-module structure on the $k$-module $\mathfrak{M}(x)$.
    \item A \emph{morphism $\varphi: \mathfrak{M} \to \mathfrak{N}$ of presheaves of $\mathfrak{R}$-modules} is a morphism of presheaves of $k$-modules $\varphi: \mathfrak{M} \to \mathfrak{N}$ such that the diagram
    $$
    \xymatrix{
    \mathfrak{R} \times \mathfrak{M} \ar[r] \ar[d]_{id \times \varphi} & \mathfrak{M} \ar[d]^{\varphi} \\
    \mathfrak{R} \times \mathfrak{N} \ar[r] & \mathfrak{N}
     } 
    $$ 
    commutes.
    \item The \emph{category of presheaves of $\mathfrak{R}$-modules} is denoted by $\PMod \mathfrak{R}$.
\end{enumerate}
\end{definition}

Now, it is time to introduce the main object that we will investigate in this paper, that is, the sheaves of modules on ringed sites.

\begin{definition} (\cite[Section 18.10, tag 03CW]{StaPro}) Let $(\C, \mathfrak{R})$ be a ringed site.
\begin{enumerate}
\item A \emph{sheaf of $\mathfrak{R}$-modules} is a presheaf of $\mathfrak{R}$-modules $\mathfrak{M}$, such that the underlying presheaf of $k$-modules $\mathfrak{M}$ is a sheaf.
\item A \emph{morphism of sheaves of $\mathfrak{R}$-modules} is a morphism of presheaves of $\mathfrak{R}$-modules. 
\item The \emph{category of sheaves of $\mathfrak{R}$-modules} is denoted by $\Mod\mathfrak{R}$.  
\end{enumerate}	
\end{definition}	

The following Proposition can be extracted from \cite[the proof of Theorem 4.2.1]{WX23}, which says that, for a finite category $\calC$, the sheaves of $\mathfrak{R}$-modules are equivalent to the presheaves of $\mathfrak{R}|_{\calD}$-modules for some subcategory $\calD \subseteq \calC$, where $\calD$ is a strictly full subcategory of $\calC$ which determines the \emph{subcategory topology} $\calJ^{\calD}=\calJ$ of $\calC$. For the definition of subcategory topology $\calJ^{\calD}$, one can see \cite[Definition 4.1.1]{WX23}. Here $\mathfrak{R}|_{\calD}$ is the restriction of $\mathfrak{R}$ to $\calD$.

\begin{proposition} \label{shispre}
Let $\calC$ be a finite category and let $\mathfrak{R}:\calC^{\rm op} \to k\mbox{\rm -Alg}$ be a sheaf of unital $k$-algebras on a site $\C=(\calC, \calJ)$. Then we have the following equivalence 
	$$
	\Mod \mathfrak{R} \simeq \PMod \mathfrak{R}|_{\calD},
	$$
for some strictly full subcategory $\calD$ of $\calC$ such that $\calJ^{\calD}=\calJ$.
\end{proposition}

\subsubsection{Additive sheaf theory}

In this subsection, we will recall some basic knowledge about additive sheaf theory, the reader is referred to \cite{Low04,Low16}.    

Firstly, we will recall the linear counterpart of Definition \ref{grotop}, namely, the linear Grothendieck topology.  

\begin{definition} (\cite[Definition 2.1]{Low04}) \label{lintop}
 Let $\mathcal{A}$ be a preadditive category. A \emph{linear Grothendieck topology} on $\mathcal{A}$ is given by specifying, for every object $x$ in $\mathcal{A}$, a collection $\calJ(x)$ of subfunctors of $\Hom_{\mathcal{A}}(-, x)$ in $\rMod \calA$ satisfying the following axioms:
 \begin{enumerate}
     \item $\Hom_{\mathcal{A}}(-, x) \in \calJ(x)$;
     \item for $S \in \calJ(x)$ and $f: y \to x$ in $\mathcal{A}$, the pullback $f^{*}S$ in $\rMod \calA$ of $S$ along $f: \Hom_{\mathcal{A}}(-, y) \to \Hom_{\mathcal{A}}(-, x)$ is in $\calJ(y)$;
     \item for $S_1 \in \calJ(x)$ and an arbitrary subfunctor $S_2$ of $\Hom_{\mathcal{A}}(-, x)$, if for every $f: y \to x \in S_1$, the pullback $f^{*}S_2$ is in $\calJ(y)$, it follows that $S_2$ is in $\calJ(x)$.
 \end{enumerate}
\end{definition}

\begin{remark}
\begin{enumerate}
    \item For a sieve $S \hookrightarrow \Hom(-,x)$ on $x$ in a linear Grothendieck topology, we can also treat it both as a subfunctor and a set-theoretic union of morphisms (\cite[Convention 2.2]{Low04}). The difference is that $S(y)$ is a \emph{subset} of $\Hom(y,x)$ in non-linear Grothendieck topology while in linear Grothendieck topology, $S(y)$ is a \emph{subgroup} of $\Hom(y,x)$.
    \item If $\mathcal{A}$ is a single object category associated to a ring $k$, then the axioms in the Definition \ref{lintop} above correspond to those of a Gabriel topology on $k$ \cite{Gab62}.
\end{enumerate}    
\end{remark}

The following notions are the linear counterpart of sites and sheaves.

\begin{definition} (\cite[Section 2.3]{Low16})
A small $k$-linear category endowed with a $k$-linear topology $\calJ$ is called a $k$-\emph{linear site}.   
\end{definition}

\begin{definition} (\cite[Definition 2.5]{Low04})
Let $\mathcal{A}$ be a small preadditive category and $\calJ$ be a topology. A presheaf $\mathfrak{F}$ in $\rMod \calA$ is called a \emph{sheaf} (resp. a \emph{separated sheaf}) if for every $\gamma: S \to \Hom_{\mathcal{A}}(-, x)$ in $\calJ$ and for every natural transformation $\alpha: S \to \mathfrak{F}$, there is a unique (resp. there is at most one) $\beta: \Hom_{\mathcal{A}}(-, x) \to \mathfrak{F}$ with $\beta \gamma = \alpha$.   
\end{definition}

We recall the definition of torsion modules below which will be used to characterize the category of the sheaves of modules on ringed sites.

\begin{definition} (\cite[page: 183]{Low16}) \label{tormod}
Let $\mathcal{A}$ be a $k$-linear category.
\begin{enumerate}
    \item A module $\mathfrak{F} \in \rMod \calA$ is \emph{torsion} if for every morphism $\alpha: \Hom_{\mathcal{A}}(-, x) \to \mathfrak{F}$, there is a cover $S \in \calJ(x)$ on which $\alpha$ vanishes, that is, the composition $\alpha|_{S}: S \to \Hom_{\mathcal{A}}(-, x) \to \mathfrak{F}$ is equal to zero. 
    \item The \emph{full subcategory of torsion modules} is denoted by ${\rm Tors}(\mathcal{A}, \calJ) \subseteq \rMod \calA$. 
\end{enumerate}
\end{definition}

\subsection{Grothendieck constructions and skew category algebras}

In this subsection, we will recall the definitions of \emph{linear} Grothendieck constructions and skew category algebras. 

Given an oplax functor, one can define the linear Grothendieck construction of it, see \cite[Definition 2.1]{AK13} for its explicit definition. We can restrict this construction to a presheaf (because a presheaf can be viewed as an oplax functor with the 2-morphisms are just identities, see \cite[Remark 2.2]{Asa13a}). Since we assume all the algebras and algebra homomorphisms in the category  $k\mbox{\rm -Alg}$ are unital, so by the work of Mitchell (see \cite{Mit72}), we know that for a presheaf of unital $k$-algebras $\mathfrak{R}:\calC^{\rm op} \to k\mbox{\rm -Alg}$ and $x \in \Ob \calC$, the $k$-algebra $\mathfrak{R}(x)$ can be viewed as a category with a single object. Thus we have the following definition.

\begin{definition} \label{lingrocon}
 Let $\calC$ be a small category, $\mathfrak{R}:\calC^{\rm op} \to k\mbox{\rm -Alg}$ be a presheaf of unital $k$-algebras on $\calC$. Then a category $Gr(\mathfrak{R})$, called the \emph{linear Grothendieck construction} of $\mathfrak{R}$, is defined as follows:
\begin{enumerate}
    \item $\Ob Gr(\mathfrak{R}):=\{(x, \bullet_x)\ |\ x \in \Ob \calC\ \mbox{and}\ \bullet_x \in \Ob \mathfrak{R}(x) \} = \Ob \calC$;
    \item for each $x, y \in \Ob Gr(\mathfrak{R})$,
    $$
    \Hom_{Gr(\mathfrak{R})}(x, y):=\bigoplus_{f \in \Hom_{\calC}(x, y)} \mathfrak{R}(x)\left(\bullet_x, \mathfrak{R}(f)(\bullet_y)\right)=\bigoplus_{f \in \Hom_{\calC}(x, y)} \mathfrak{R}(x);
    $$
    \item for each $x, y, z \in \Ob Gr(\mathfrak{R})$, and each $(r_f)_{f \in \Hom_{\calC}(x, y)} \in \Hom_{Gr(\mathfrak{R})}(x, y)$, $(s_g)_{g \in \Hom_{\calC}(y, z)} \in \Hom_{Gr(\mathfrak{R})}(y, z)$, we set
    $$
    (s_g)_g \circ (r_f)_f := \left( \sum_{\substack{f \in \Hom_{\calC}(x, y)\\ g \in \Hom_{\calC}(y, z)\\ h=gf}} \mathfrak{R}(f)(s_g) r_f \right)_{h \in \Hom_{\calC}(x, z),}
    $$
where each summand is the composite
$$
\xymatrix{
\bullet_x \ar[r]^(.35){r_f} &  \mathfrak{R}(f)(\bullet_y) \ar[r]^(.3){\mathfrak{R}(f)(s_g)} & \mathfrak{R}(f)(\mathfrak{R}(g)(\bullet_z))=\mathfrak{R}(gf)(\bullet_z).
}
$$
\end{enumerate}
\end{definition}

\begin{remark} \label{rmk}
    \begin{enumerate}
        \item For a presheaf $\mathfrak{R}:\calC^{\rm op} \to k\mbox{\rm -Alg}$ of unital $k$-algebras, its linear Grothendieck construction is a preadditive category, since one can immediately check that $\Hom_{Gr(\mathfrak{R})}(x, y)$ is an Abelian group with addition given as follows
        $$
        (r_f)_f + (r'_f)_f := (r_f + r'_f)_f,
        $$
        for any $(r_f)_f, (r'_f)_f \in \Hom_{Gr(\mathfrak{R})}(x, y)$. However, $Gr(\mathfrak{R})$ is usually not an additive category, because it may not have zero object.
        \item The linear Grothendieck construction $Gr(\mathfrak{R})$ above is different (that is why we use different notations) from the Grothendieck construction $Gr_{\calC}\mathfrak{R}$ in \cite[the paragragh after Definition 3.1.1]{WX23}, the Grothendieck construction $Gr_{\calC}\mathfrak{R}$ is not a preadditive category, see the comment before \cite[Section 3.2]{WX23}. 
        \item The \emph{semi-direct product} construction $\mathfrak{R} \bigotimes_{\theta} \mathbf{A}$ of a small site $\mathbf{A}$ by a sheaf of rings $\mathfrak{R}$ on it (see \cite[Section 3]{How81}) is just the linear Grothendieck construction $Gr(\mathfrak{R})$ of $\mathfrak{R}$.
    \end{enumerate}
\end{remark}

Motivated by the (non-linear) Grothendieck construction, in \cite{WX23}, we introduced the notion of skew category algebras $\mathfrak{R}[\calC]$. For the convenience of the reader, we record its definition below.

\begin{definition} (\cite[Definition 3.2.1]{WX23}) \label{skewcatalg}
	Let $\calC$ be a (non-empty) small category. Let $\mathfrak{R}:\calC^{\rm op} \to k\text{-}{\rm Alg}$ be a presheaf of $k$-algebras. The \emph{skew category algebra} $\mathfrak{R}[\calC]$ on $\calC$ with respect to $\mathfrak{R}$ is a $k$-module spanned over elements of the form $rf$, where $f \in \Mor \calC$ and $r \in \mathfrak{R}(\dom(f))$. We define the multiplication on two base elements by the rule
	\begin{eqnarray}
		sg \ast rf=
		\begin{cases}
			(\mathfrak{R}(f)(s)r) gf,       & \text{if}\ \dom(g)=\cod(f); \notag \\
			0, & {\rm otherwise}.
		\end{cases}
	\end{eqnarray} 
Extending this product linearly to two arbitrary elements, $\mathfrak{R}[\calC]$ becomes an associative $k$-algebra.
\end{definition}

\subsection{Torsion pairs, TTF triples and Abelian recollements}
In this subsection, the definitions of torsion pairs, TTF triples and Abelian recollements will be recalled.  

\subsubsection{Torsion pairs}
Torsion theories (also called torsion pairs) were introduced by Dickson \cite{Dic66} in general setting of Abelian categories, taking as a model the classical theory of torsion Abelian groups. Torsion pairs have played an important role in studying the Grothendieck categories and their localizations.

\begin{definition}  \label{torpairdef} 
Let $\mathcal{A}$ be an Abelian category. A \emph{torsion pair} (or \emph{torsion theory}) in $\mathcal{A}$ is a pair of full subcategories $(\mathcal{X}, \mathcal{Y})$ such that:

\begin{enumerate}
    \item $\mathcal{X}=~^{\perp_0}\mathcal{Y}$ and $\mathcal{Y}=\mathcal{X}^{\perp_0}$, where
    $$
    ^{\perp_0}\mathcal{Y}:=\{ a \in \Ob \mathcal{A}\ |\ \Hom_{\mathcal{A}}(a, y)=0, \forall\ y \in \Ob \mathcal{Y} \},
    $$
    $$
    \mathcal{X}^{\perp_0}:=\{ a \in \Ob \mathcal{A}\ |\ \Hom_{\mathcal{A}}(x, a)=0, \forall\ x \in \Ob \mathcal{X} \}.
    $$
    \item for each $a \in \Ob \mathcal{A}$, there is an exact sequence 
    \begin{align}\label{sqe}
    0 \to x_a \to a \to y_a \to 0, \tag{\dag}
    \end{align}
    with $x_a \in \Ob \mathcal{X}$ and $y_a \in \Ob \mathcal{Y}$.
\end{enumerate}
\end{definition}

If $(\mathcal{X}, \mathcal{Y})$ is a torsion pair in $\mathcal{A}$, then $\mathcal{X}$ is called a \emph{torsion class} and $\mathcal{Y}$ is called a \emph{torsion-free class}. A subcategory of $\mathcal{A}$ is the torsion class (resp. torsion-free class) of some torsion pair if and only if it is closed under quotients, direct sums and extensions (resp. subobjects, direct products and extensions, respectively), see \cite[Theorem 2.3]{Dic66}. 

There are some special kinds of torsion pairs will be considered in this paper.

\begin{definition}
\begin{enumerate}
    \item A torsion pair $(\mathcal{X}, \mathcal{Y})$ in $\mathcal{A}$ is called \emph{hereditary} if $\mathcal{X}$ is also closed under subobjects.
    \item A torsion pair $(\mathcal{X}, \mathcal{Y})$ is said to be \emph{split} if, for any object $x \in \Ob \mathcal{X}$, the canonical sequence (\ref{sqe}) splits.
\end{enumerate}
\end{definition}

\subsubsection{TTF triples}
In 1965, Jans \cite{Jan65}  first introduced the notion of torsion torsion-free theory which is now called a \emph{TTF triple}. In this subsection, we will recall the definitions of TTF triples and split TTF triples, as well as the TTF triples generated by a class of objects.

\begin{definition}
\begin{enumerate}
    \item A \emph{TTF triple} in an Abelian category $\mathcal{A}$ is a triple of full subcategories $(\mathcal{X}, \mathcal{Y}, \mathcal{Z})$ such that both $(\mathcal{X}, \mathcal{Y})$ and $(\mathcal{Y}, \mathcal{Z})$ are torsion pairs. 
    \item A TTF triple $(\mathcal{X}, \mathcal{Y}, \mathcal{Z})$ is said to \emph{split} if both torsion pairs $(\mathcal{X}, \mathcal{Y})$ and $(\mathcal{Y}, \mathcal{Z})$ split.
\end{enumerate}   
\end{definition}

\begin{definition} (\cite[Definition 4.6]{PSV21})
Let $(\mathcal{X}, \mathcal{Y}, \mathcal{Z})$ be a TTF triple in $\rMod \calA$ and let $\mathcal{S}$ be a class of right $\mathcal{A}$-modules. We say that $(\mathcal{X}, \mathcal{Y}, \mathcal{Z})$ is generated by $\mathcal{S}$ when the torsion pair $(\mathcal{X}, \mathcal{Y})$ is \emph{generated by $\mathcal{S}$}, i.e., when $ \mathcal{Y}=\mathcal{S}^{\perp_0}$ (the meaning of $(-)^{\perp_0}$ here is the same as that of Definition \ref{torpairdef}). Furthermore, we say that $(\mathcal{X}, \mathcal{Y}, \mathcal{Z})$ is \emph{generated by finitely generated projective $\mathcal{A}$-modules} when it is generated by a set of finitely generated projective objects of $\rMod \calA$.    
\end{definition}

\subsubsection{Abelian recollements}
There is an alternative way to think about the TTF triples, namely, the \emph{Abelian recollements}. Recollements were first introduced in the context of triangulated categories by Beilinson, Bernstein and Deligne \cite{BBD82}. A fundamental example of a recollement of Abelian categories appeared in the construction of perverse sheaves by MacPherson and Vilonen \cite{MV86}. For more information about Abelian recollements, one is referred to \cite{FP04}. Now, for the convenience of the reader, we will record the definition of the Abelian recollement of $\rMod \calA$ below.

\begin{definition}
 Let $\mathcal{A}$ be a small preadditive category. A \emph{recollement} $\mathscr{R}$ of $\rMod \calA$ by Abelian categories $\mathcal{X}$ and $\mathcal{Y}$ (also called an \emph{Abelian recollement}) is a diagram of additive functors
$$
\xymatrix{
\mathscr{R}:\ \mathcal{Y} \ar[rr]|{i_*} & & \rMod \calA \ar@/^1pc/[ll]^{i^!} \ar@/_1pc/[ll]_{i^*}  \ar[rr]|{j^*} & & \mathcal{X} \ar@/^1pc/[ll]^{j_*} \ar@/_1pc/[ll]_{j_!}
} 
$$ 
satisfying the following conditions:
\begin{enumerate}
    \item $(i^*, i_*, i^!)$ and $(j_!, j^*, j_*)$ are adjoint triples;
    \item the functors $i_*$, $j_!$ and $j_*$ are fully faithful;
    \item ${\rm Im}(i_*)={\rm Ker}(j^*)$.
\end{enumerate}
\end{definition}

Two Abelian recollements $\mathscr{R}: (\mathcal{Y}, \rMod \calA, \mathcal{X})$ and $\mathscr{R}': (\mathcal{Y}', \rMod \calA, \mathcal{X}')$ of $\rMod \calA$ are said to be \emph{equivalent} if there are equivalences $\Phi: \rMod \calA \to \rMod \calA$ and $\Psi: \mathcal{X} \to \mathcal{X}'$ such that the following diagram commutes, up to natural isomorphism:
$$
\xymatrix{
\rMod \calA \ar[r]^(.6){j^*} \ar[d]_{\Phi} & \mathcal{X} \ar[d]^{\Psi} \\
\rMod \calA \ar[r]^(.6){(j^*)'} & \mathcal{X}'.
} 
$$

\section{A characterization of the categories of modules on ringed sites} \label{torperp}

In this section, we will firstly give another proof of the \cite[Theorem A]{WX23}. Then a new characterization, in terms of the torsion modules, of the category of sheaves of modules on ringed sites will be given.

Let $\calC$ be a small category and let $\mathfrak{R}:\calC^{\rm op} \to k\mbox{\rm -Alg}$ be a sheaf of unital $k$-algebras on $\C$. We already knew that the semi-direct product construction $\mathfrak{R} \bigotimes_{\theta} \mathbf{A}$ of a small site $\mathbf{A}$ by a sheaf of rings $\mathfrak{R}$ on it (see \cite[Section 3]{How81}) is just the linear Grothendieck construction $Gr(\mathfrak{R})$ of $\mathfrak{R}$, see Remark \ref{rmk} (3). The following Theorem tells us that the sheaves of $\mathfrak{R}$-modules can be viewed as the Abelian group-valued sheaves on $Gr(\mathfrak{R})$.

\begin{theorem} \label{How81} (see \cite[Proposition 5]{How81} and \cite[Theorem 5]{Mur06B})
	Let $\calC$ be a small category and let $\mathfrak{R}:\calC^{\rm op} \to k\mbox{\rm -Alg}$ be a sheaf of unital $k$-algebras on a site $\C=(\calC, \calJ)$. Then we have the following equivalence 
	$$
	\Mod \mathfrak{R} \simeq \Sh((Gr(\mathfrak{R}), \calJ'), \Ab),
	$$
where the linear Grothendieck topology $\calJ'$ on $Gr(\mathfrak{R})$ is the ``$\mathfrak{R}$-linearization'' of the Grothendieck topology $\calJ$ on $\calC$. In particular, when $\calJ$ is the trivial topology, we have
    $$
	\PMod \mathfrak{R} \simeq \rMod Gr(\mathfrak{R}).
	$$
\end{theorem}

\begin{proof}[Sketch of the proof]
Firstly, let us define a functor 
     $$
     \Phi: \Sh((Gr(\mathfrak{R}), \calJ'), \Ab) \to \Mod\mathfrak{R}
     $$ 
by
	$$
	\mathfrak{F} \mapsto \mathfrak{M}_{\mathfrak{F}},
	$$ 
 where $\mathfrak{M}_{\mathfrak{F}}$ is defined as follows:
 for each $f \in \Hom_{\calC}(x, y)$,
 $$
\xymatrix{
\mathfrak{M}_{\mathfrak{F}}(f): \mathfrak{F}(y) \ar[rrrr]^(.55){\mathfrak{F}[(\cdots, 0_{\mathfrak{R}(x)}, 1_{\mathfrak{R}(x)f},0_{\mathfrak{R}(x)}, \cdots)]} & & & &\mathfrak{F}(x),
}
$$
where $0_{\mathfrak{R}(x)}, 1_{\mathfrak{R}(x)}$ are the zero and identity elements of the $k$-algebra $\mathfrak{R}(x)$ respectively, and $(\cdots, 0_{\mathfrak{R}(x)}, 1_{\mathfrak{R}(x)f},0_{\mathfrak{R}(x)}, \cdots) \in \Hom_{Gr(\mathfrak{R})}(x, y)$.

We also have to define an $\mathfrak{R}$-action on $\mathfrak{M}_{\mathfrak{F}}$:
    for each $r \in \mathfrak{R}(x)$ and each $m \in \mathfrak{M}_{\mathfrak{F}}(x) = \mathfrak{F}(x)$,
    $$
    m \cdot r := \mathfrak{F}[(\cdots, 0_{\mathfrak{R}(x)}, r_{1_x},0_{\mathfrak{R}(x)}, \cdots)](m),
    $$
where $(\cdots, 0_{\mathfrak{R}(x)}, r_{1_x},0_{\mathfrak{R}(x)}, \cdots) \in \Hom_{Gr(\mathfrak{R})}(x, x)$.

Secondly, let us define another functor 
    $$
    \Psi: \Mod \mathfrak{R} \to \Sh((Gr(\mathfrak{R}), \calJ'), \Ab)
    $$ 
by
	$$
	\mathfrak{M} \mapsto \mathfrak{F}_{\mathfrak{M}},
	$$ 
 where $\mathfrak{F}_{\mathfrak{M}}$ is defined as follows:
 for each $(r_f)_f \in \Hom_{Gr(\mathfrak{R})}(x, y)$,
 $$
\xymatrix{
\mathfrak{F}_{\mathfrak{M}}[(r_f)_f]: \mathfrak{M}(y) \ar[rrr]^(.6){\sum_{r, f} (- \cdot r) \circ \mathfrak{M}(f)} & &  &\mathfrak{M}(x).
}
$$
Here, the map $(- \cdot r) \circ \mathfrak{M}(f)$ is depicted as below
$$
\xymatrix{
\mathfrak{M}(y) \ar[r]^{\mathfrak{M}(f)} & \mathfrak{M}(x) \ar[r]^{(- \cdot r)} & \mathfrak{M}(x), \\
}
$$
and
$(- \cdot r)$ means the right action of $r \in \mathfrak{R}(x)$ on $m \in \mathfrak{M}(x)$.

One can check that both the functors $\Phi$ and $\Psi$ are well defined, and they give rise to the desired equivalence.
\end{proof}

\begin{remark}
For the explicit definition of the linear Grothendieck topology $\calJ'$ on $Gr(\mathfrak{R})$ above, one can see \cite[Theorem 5]{Mur06B}.   
\end{remark}

Thanks to the Theorem \ref{How81} above, we can now give another proof of the \cite[Theorem A]{WX23} as follows. 

\begin{theorem} (\cite[Theorem A]{WX23}) \label{reproof}
	Let $\calC$ be a small category and let $\mathfrak{R}:\calC^{\rm op} \to k\mbox{\rm -Alg}$ be a presheaf of unital $k$-algebras on $\calC$. If $\Ob \calC$ is finite, then we have a category equivalence 
	$$
	\PMod \mathfrak{R} \simeq \rMod \mathfrak{R}[\calC].
	$$  
\end{theorem}

\begin{proof}
By Theorem \ref{How81}, we have the following equivalence
    $$
	\PMod \mathfrak{R} \simeq \rMod Gr(\mathfrak{R}).
	$$
 
Because $\Ob \calC$ is finite, hence $\Ob Gr(\mathfrak{R})$ is finite too. Let 
   $$
   \mathfrak{G}:=\bigoplus_{x \in \Ob Gr(\mathfrak{R})}\Hom_{Gr(\mathfrak{R})}(-, x),
   $$
then one can immediately see that $\mathfrak{G}$ is a compact projective generator of the category $\rMod Gr(\mathfrak{R})$. Therefore, by the Freyd-Mitchell theorem (see \cite[Exercise F on page 106]{Fre64}), we have the following equivalence
    $$
    \rMod Gr(\mathfrak{R}) \simeq \rMod \End(\mathfrak{G}).
    $$

 Now, let us compute $\End(\mathfrak{G})$ \footnote{In the early version, the compactness of the first argument of Hom functor was used to prove the first isomorphism. The author would like to thank the reviewer for the comment that it is actually unnecessary to do so, since  finite products are the same as the finite coproducts and the co-variant Hom functor preserves limits, especially products.}:
\begin{align*}
     & \End(\mathfrak{G}) \\
   = & \Hom\left(\bigoplus_{x \in \Ob Gr(\mathfrak{R})}\Hom_{Gr(\mathfrak{R})}(-, x), \bigoplus_{y \in \Ob Gr(\mathfrak{R})}\Hom_{Gr(\mathfrak{R})}(-, y)\right)\\
   \cong & \bigoplus_{y \in \Ob Gr(\mathfrak{R})} \Hom\left(\bigoplus_{x \in \Ob Gr(\mathfrak{R})}\Hom_{Gr(\mathfrak{R})}(-, x), \Hom_{Gr(\mathfrak{R})}(-, y)\right) & (fin. prod. = fin. copr.) \\
    \cong & \bigoplus_{y \in \Ob Gr(\mathfrak{R})} \bigoplus_{x \in \Ob Gr(\mathfrak{R})}\Hom\left(\Hom_{Gr(\mathfrak{R})}(-, x), \Hom_{Gr(\mathfrak{R})}(-, y)\right) & (\Ob Gr(\mathfrak{R})\ is\ finite)\\
     \cong & \bigoplus_{y \in \Ob Gr(\mathfrak{R})} \bigoplus_{x \in \Ob Gr(\mathfrak{R})}\Hom_{Gr(\mathfrak{R})}(x, y) & (Yoneda)\\
     \cong & \mathfrak{R}[\calC] & {\left( \Phi:  (r_f)_f \mapsto \sum_f rf \right)}.  
\end{align*}
It is not hard to see that the map $\Phi$ above is a bijection. The map $\Phi$ is an isomorphism because it also preserves the multiplication:
\begin{align*}
     & \Phi((s_g)_g \circ (r_f)_f) & \\ 
    = & \Phi \left( \left( \sum_{h=gf} \mathfrak{R}(f)(s)r \right)_h \right) & (Definition~\ref{lingrocon}~(3)) \\
    = & \sum_h \left( \sum_{h=gf} \mathfrak{R}(f)(s)r \right)h & (Definition~of~\Phi) \\
    = & \sum_h \left( \sum_{h=gf} (\mathfrak{R}(f)(s)r)h \right) & (addition~in~\mathfrak{R}[\calC]) \\
    = & \sum_h \left( \sum_{h=gf} sg  \ast rf \right) & (Definition~\ref{skewcatalg}) \\
    = & \left( \sum_g sg \right) \ast \left( \sum_f rf \right) & (Definition~of~\Phi) \\
    = & \Phi((s_g)_g) \ast \Phi((r_f)_f). & \\
\end{align*}

Thus, together with Theorem \ref{How81}, we have
    $$
   \PMod \mathfrak{R} \simeq \rMod Gr(\mathfrak{R}) \simeq \rMod \mathfrak{R}[\calC].
	$$ 
 This completes the proof.
\end{proof}

Using Theorem \ref{How81}, we can also get a new characterization of the category of sheaves of $\mathfrak{R}$-modules by the torsion modules on $Gr(\mathfrak{R})$.

\begin{theorem} \label{tor}
Let $\calC$ be a small category and let $\mathfrak{R}:\calC^{\rm op} \to k\mbox{\rm -Alg}$ be a sheaf of unital $k$-algebras on a site $\C=(\calC, \calJ)$. Then we have the following equivalence  
    $$
  \Mod \mathfrak{R} \simeq {\rm Tors}(Gr(\mathfrak{R}), \calJ')^\perp,
    $$
 for some linear Grothendieck topology $\calJ'$ on $Gr(\mathfrak{R})$. 
\end{theorem}

\begin{proof}
 It follows from Theorem \ref{How81} and \cite[Proposition 2.7]{Low16}.  
\end{proof}

\begin{remark}
In \cite[Theorem 1.1(2) and Remark 1.2]{DLLX22}, the authors study the sheaves of modules on atomic sites, and they show that the category of sheaves is equivalent to the Serre quotient of the category of presheaves by the category of torsion presheaves. But the Theorem \ref{tor} above says that it is actually true for \emph{all} Grothendieck topologies, not just atomic topology \footnote{Recently, the author noticed that Prof. Liping Li 
\begin{CJK*}{UTF8}{}
\CJKtilde \CJKfamily{gbsn}(李利平)   
\end{CJK*} 
has already got a characterization of sheaves of $\mathfrak{R}$-modules, for \emph{any} Grothendieck topology, in terms of a certain torsion theory. One can see the abstracts of his talks: \href{https://slxy.lzjtu.edu.cn/info/1064/2076.htm}{``\emph{A torsion theoretic interpretation of sheaf theory}''} and \href{https://math.qfnu.edu.cn/info/1200/7116.htm}{``\emph{A torsion theoretic interpretation of sheaves over ringed sites}''}.}.      
\end{remark}

\section{Torsion theory of the category of modules on ringed finite sites} \label{classification}

In this section, we will classify the hereditary torsion pairs, (split) TTF triples and Abelian recollements of the category $\Mod \mathfrak{R}$ of modules on ringed finite sites.

In the category $\rMod \calA$ of right $\mathcal{A}$-modules, the hereditary torsion pairs can be parameterized by the linear Grothendieck topologies on $\mathcal{A}$ (see \cite[Theorem 3.7]{PSV21}). Our first result of this section indicates that the hereditary torsion pairs in the category of modules on ringed finite sites can still be classified by the linear Grothendieck topologies. 

\begin{theorem} \label{htp}
Let $\calC$ be a small category and let $\mathfrak{R}:\calC^{\rm op} \to k\mbox{\rm -Alg}$ be a presheaf of unital $k$-algebras on $\calC$. Then there is an (explicit) one-to-one correspondence between linear Grothendieck topologies on $Gr(\mathfrak{R})$ and hereditary torsion pairs in $\PMod \mathfrak{R}$. Moreover, if $\calC$ is finite and $\mathfrak{R}:\calC^{\rm op} \to k\mbox{\rm -Alg}$ is a sheaf of unital $k$-algebras on a site $\C=(\calC, \calJ)$, then there is an (explicit) one-to-one correspondence between linear Grothendieck topologies on $Gr(\mathfrak{R}|_{\mathcal{D}})$ and hereditary torsion pairs in $\Mod \mathfrak{R}$, where $\mathcal{D}$ is a strictly full subcategory of $\mathcal{C}$ such that $\mathcal{J}^{\mathcal{D}}=\mathcal{J}$ and $\mathfrak{R}|_{\mathcal{D}}$ is the restriction of $\mathfrak{R}$ to $\calD$.    
\end{theorem}

\begin{proof}
By Theorem \ref{How81}, we have
    $$
	\PMod \mathfrak{R} \simeq \rMod Gr(\mathfrak{R}).
	$$
Then, by \cite[Theorem 3.7]{PSV21}, there is an explicit one-to-one correspondence between linear Grothendieck topologies on $Gr(\mathfrak{R})$ and hereditary torsion pairs in $\rMod Gr(\mathfrak{R})$. So we get the first part of our theorem. 

Since $\calC$ is finite, by Proposition \ref{shispre}, we know that 
    $$
	\Mod \mathfrak{R} \simeq \PMod \mathfrak{R}|_{\calD},
	$$
for some strictly full subcategory $\calD$ of $\calC$ such that $\calJ^{\calD}=\calJ$. Then we can apply Theorem \ref{How81} and \cite[Theorem 3.7]{PSV21} to the category $\PMod \mathfrak{R}|_{\calD}$ of presheaves of $\mathfrak{R}|_{\calD}$-modules. This completes the proof. 
\end{proof}

In the category of $\rMod \calA$ of right $\mathcal{A}$-modules, the TTF triples are parameterized by the idempotent ideals of $\mathcal{A}$ (see \cite[Theorem 4.5]{PSV21}). Our second result of this section says that one can use idempotent ideals to classify the TTF triples in the category of modules on ringed finite sites too.

\begin{theorem} \label{ttf}
Let $\calC$ be a small category and let $\mathfrak{R}:\calC^{\rm op} \to k\mbox{\rm -Alg}$ be a presheaf of unital $k$-algebras on $\calC$. Then there is an (explicit) one-to-one correspondence between idempotent ideals of $Gr(\mathfrak{R})$ and TTF triples in $\PMod \mathfrak{R}$. Moreover, if $\calC$ is finite and $\mathfrak{R}:\calC^{\rm op} \to k\mbox{\rm -Alg}$ is a sheaf of unital $k$-algebras on a site $\C$, then there is an (explicit) one-to-one correspondence between idempotent ideals of $Gr(\mathfrak{R}|_{\calD})$ and the TTF triples in $\Mod \mathfrak{R}$ for some strictly full subcategory $\calD$ of $\calC$.      
\end{theorem}

\begin{proof}
By Theorem \ref{How81}, we have
    $$
	\PMod \mathfrak{R} \simeq \rMod Gr(\mathfrak{R}).
	$$
 Using \cite[Theorem 4.5]{PSV21}, there is an explicit one-to-one correspondence between idempotent ideals of $Gr(\mathfrak{R})$ and TTF triples in $\rMod Gr(\mathfrak{R})$. So we have proved the first part of the theorem. The second part's argument is similar to that of the Theorem \ref{htp} above. 
\end{proof}

In \cite[Theorem D and Proposition 4.13]{PSV21}, the authors tell us that the split TTF triples in the category of $\rMod \calA$ of right $\mathcal{A}$-modules are one-to-one correspond to the idempotents of the center $Z(\mathcal{A})$ of $\mathcal{A}$, it seems that the similar bijection should be exist in the category of modules on ringed finite sites. That is exactly what our third result says.

\begin{theorem} \label{sttf}
Let $\calC$ be a small category and let $\mathfrak{R}:\calC^{\rm op} \to k\mbox{\rm -Alg}$ be a presheaf of unital $k$-algebras on $\calC$. Then there is an (explicit) one-to-one correspondence between idempotents of the center $Z(Gr(\mathfrak{R}))$ of $Gr(\mathfrak{R})$ and split TTF triples in $\PMod \mathfrak{R}$. Moreover, if $\calC$ is finite and $\mathfrak{R}:\calC^{\rm op} \to k\mbox{\rm -Alg}$ is a sheaf of unital $k$-algebras on a site $\C$, we can also classify the split TTF triples in $\Mod \mathfrak{R}$ by the idempotents of the center $Z(Gr(\mathfrak{R}|_{\calD}))$ of $Gr(\mathfrak{R}|_{\calD})$ for some strictly full subcategory $\calD$ of $\calC$.     
\end{theorem}

\begin{proof}
Using Theorem \ref{How81} together with \cite[Theorem D and Proposition 4.13]{PSV21}, we will obtain the first part. The second part use the same trick as Theorem \ref{htp}.  
\end{proof}

The last result is devoted to investigate the Abelian recollements in the category of modules on ringed finite sites. Motivating by the \cite[Theorem 4.8]{PSV21}, we obtain a classification of the Abelian recollements in the category of modules on ringed finite sites as follows.  

\begin{theorem} \label{ar}
There are (explicit) one-to-one correspondences between:
\begin{enumerate}
    \item the equivalence classes of recollements of $\PMod \mathfrak{R}$ by module categories over small preadditive categories;
    \item the TTF triples in $\PMod \mathfrak{R}$ generated by finitely generated projective $Gr(\mathfrak{R})$-modules;
    \item the idempotent ideals of $Gr(\mathfrak{R})$ which are the trace of a set of finitely generated projective $Gr(\mathfrak{R})$-modules;
    \item the idempotent ideals of the additive closure $\widehat{Gr(\mathfrak{R})}$ of $Gr(\mathfrak{R})$ generated by a set of idempotent endomorphisms;
    \item the full subcategory of the Cauchy completion $\widehat{Gr(\mathfrak{R})}_\oplus$ which are closed under coproducts and summands. 
\end{enumerate}
Moreover, if $\calC$ is finite and $\mathfrak{R}:\calC^{\rm op} \to k\mbox{\rm -Alg}$ is a sheaf of unital $k$-algebras on a site $\C$, then we can get the analogous one-to-one correspondences in $\Mod \mathfrak{R}$.
\end{theorem}

\begin{proof}
(1)-(5) follows from Theorem \ref{How81} and \cite[Theorem 4.8]{PSV21}. Again, the second part's proof is similar to that of Theorem \ref{htp}.     
\end{proof}

\section*{Acknowledgments}
I would like to thank my Ph.D. supervisor Prof. Fei Xu 
\begin{CJK*}{UTF8}{}
\CJKtilde \CJKfamily{gbsn}(徐斐)
\end{CJK*}
for teaching me sheaf theory and representation theory.

\printbibliography

\end{document}